\documentclass[11pt]{amsart}
\usepackage{amssymb,amsthm,mathrsfs,footnote}
\usepackage{indentfirst}
\usepackage{mathtools}
\usepackage{tikz,xcolor}

\usepackage[colorlinks=true]{hyperref}
\hypersetup{allcolors=[rgb]{0,0.5,0.5}}
\newtheorem{thm}{Theorem}[section]

\newtheorem{cor}[thm]{Corollary}

\newtheorem{lem}[thm]{Lemma}
\newtheorem{rem}[thm]{Remark}

\newtheorem{pro}[thm]{Proposition}
\newtheorem{definition}[thm]{Definition}
\newcommand\intr[1]{\mathring{#1}}

\newcommand{\suchthat}{\;\ifnum\currentgrouptype=16 \middle\fi|\;}

\newcommand{\mr}{\mathscr{R}}

\renewcommand{\Im}{\mbox{Im }}

\newcommand{\R}{\mathbb R}
\newcommand{\C}{\mathbb C}
\newcommand{\N}{\mathbb N}
\renewcommand{\H}{\mathbb H}
 
\newcommand{\toby}[1]{\stackrel{#1}{\longrightarrow}}
\newcommand{\norm}[1]{|\!|{#1}|\!|}
\def\subclassname{{\bfseries 2010 Mathematics Subject Classification
}\enspace}
\def\subclass#1{\par\addvspace\medskipamount{\rightskip=0pt plus1cm
\def\and{\ifhmode\unskip\nobreak\fi\ $\cdot$
}\noindent\subclassname\ignorespaces#1\par}}
\def\emailname{E-mail}%
\def\email#1{\emailname: #1}
\definecolor{lime}{HTML}{A6CE39}
\DeclareRobustCommand{\orcidicon}{%
	\begin{tikzpicture}
	\draw[lime, fill=lime] (0,0) 
	circle [radius=0.16] 
	node[white] {{\fontfamily{qag}\selectfont \tiny ID}};
	\draw[white, fill=white] (-0.0625,0.095) 
	circle [radius=0.007];
	\end{tikzpicture}
	\hspace{-2mm}
}

\foreach \x in {A, ..., Z}{%
	\expandafter\xdef\csname orcid\x\endcsname{\noexpand\href{https://orcid.org/\csname orcidauthor\x\endcsname}{\noexpand\orcidicon}}
}

\begin{document}

\title{On the cell structure of flag manifolds}

\author[Moncef Ghazel]{Moncef Ghazel\orcidA{}}
\date{}

\maketitle
\begin{abstract}
We here define a cell structure for real, complex and quaternionic flag manifolds in a unified way. Our method is geometric in nature and is inspired from a method due to Milnor and Stasheff, which they used to define a cell structure for real Grassmann manifolds. 
  \end{abstract}
	
	\let\thefootnote\relax\footnote{Moncef Ghazel}
 \let\thefootnote\relax\footnote{\email{moncef.ghazel@fst.utm.tn}}
 \let\thefootnote\relax\footnote{Faculté des Sciences de Tunis, University of Tunis El Manar.  Tunis, Tunisia.}
%\makeatletter{\renewcommand*{\@makefnmark}{}

%\let\thefootnote\relax\footnote{\subclass{14M15, 57Q05}}
%\subclass{14M15, 57Q05}

\section{Introduction}\label{s1}
A cell structure on a Hausdorff topological space is a partition of it into open cells in a regular way. Such a structure yields, in many interesting cases, information about the topological data associated to the space using the means of algebraic topology. More precisely, a cell structure allows for computation of Betti numbers, Poincaré polynomial and Euler characteristic. Singular (co)homology of a CW-complex is readily computable using cellular homology. To compute an extraordinary (co)homology theory for a CW-complex, one has the Atiyah-Hirzebruch spectral sequence which is the analogue of cellular homology \cite{AH2}. The Serre spectral sequence is defined using a CW-complex approximation of the base space of the fibration.\\

Perhaps one of the most important cell structures defined in the literature is that of the real and complex Grassmann manifolds given by Ehresmann \cite{E1}, for it was essentially used  in the theory of caracteristic classes developed by Pontrjagin and Chern \cite{P, C}. In a subsequent work \cite{E2}, Ehresmann extended his construction of the cell structure to the case of real flag manifolds. His definitions have certain deficiencies in the sense that they do not allow for a comprehensive proof of the axioms of a CW-complex. This is no surprise, for his work preceded that Whitehead on his celebrated combinatorial homotopy paper, where he gave the definitive definition of a CW-complex. The purpose of the present work is two fold: firstly,  to suggest a new definition of Ehresmann CW-complex structure of real fleg manifolds which allows one to give an exhaustive proof of the required axioms; and secondly, to extend this definition to the complex and quaternionic cases. Our method maybe viewed as a strict generalization of that of Milnor and Stasheff for the cell structure of real Grassmann manifolds \cite[Chapter 6]{MS}.\\

We next describe the contents of the different sections of this paper. In Section \ref{s2}, we define the notion of an open cell in a topological space and give some of its technical properties. These properties will be used in the proof of our main results. Section \ref{s3} is a preliminary one, it contains the definitions and the basic properties of Stiefel and Grassmann manifolds. In Section \ref{s4}, we briefly review the method of Milnor and Stasheff \cite[Chapter 6]{MS} for the cell structure of real Grassmann manifolds and extend it to the complex and quaternionic cases. The results of Section \ref{s4} are used as the basis for the reconstruction of the cell structure for flag manifolds which is done in the last two Sections.\\

To conclude this introduction, we observe that there is a cell structure defined for flag manifolds using Lie theoretic methods (See \cite[page 193]{B1} and  \cite{B2}) 
    
\section{Cells and CW-complexes}\label{s2}
For any topological space $X$ and any subset $A$ of $X$, let $\intr{A}$ and $\bar{A}$ denote the interior and the closure of $A$ respectively.  Let $D^n$ be the closed unit ball in the $n$-dimensional Euclidien space $\R^n$  and $S^{n-1}$ the unit sphere in $\R^n$.
\begin{definition}\label{d2}
Let $X$ be a topological space.
\begin{enumerate}
	\item A subset $e$ of $X$ is called an open cell in $X$ if there exists a continuous map
$f:D^n\longrightarrow X$ which maps $\intr{D^n}$ homeomorphically onto $e$. $f$ is then called a caracteristic map for $e$. The integer $n$ is unique, it is called the dimension of the open cell $e$ and is denoted by $\dim e$.
	\item The boundary of an open cell $e$ in $X$ is the subset  $\partial e=\bar{e}-e$.
\end{enumerate} 
\end{definition}

The next lemma will play an important role in the sequel.
\begin{lem}\label{l16}
Let $e$ be an open cell in a Hausdorff space $X$ with caracteristic map $f:D^n\longrightarrow X$. Then
\begin{enumerate}
	\item $f(D^n)=\bar{e}$.
	\item $f(\intr{D^n})\cap f(S^{n-1})=\emptyset$.
	\item $\partial e=f(S^{n-1})$; i.e. $\partial e=f(\partial \intr{D^n})$.
\end{enumerate}
 In particular, $f(D^n)$ and $f(S^{n-1})$ depend only on $e$ and not on the choice of the characteristic map $f$ of $e$.
\end{lem}
\begin{proof}
\begin{enumerate}
	\item $e\subset f(D^n)$, $f(D^n)$ is compact in $X$ which is Hausdorff, therefore $f(D^n)$ is closed, it follows that
	$\bar{e}\subset f(D^n)$. In the other hand,\\
	$f(D^n)=f(\bar{\intr{D^n}})\subset\overline{f(\intr{D^n})}=\bar{e}$; i.e. $f(D^n)\subset\bar{e}$. Thus $\bar{e}=f(D^n)$.
	\item Let $g:\intr{D^n}\longrightarrow e$, be the homeomorphism induced by $f$. Let $y_0\in e$, $V$ an open neighborhood of $y_0$ in the subspace $e$ of $X$ with compact closure $\bar{V}$ in $e$.	
	$g^{-1}(\bar V)$ is compact in $\intr{D^n}$, there exists $r_0\in(0,1)$ such that $g^{-1}(\bar V)\subset B(0,r_0)$ where
	$B(0,r_0)$ is the open ball of center $0$ and radius $r_0$.
	$g^{-1}(\bar V)\subset B(0,r_0)$, thus
	\begin{eqnarray*}
\intr{D^n}-B(0,r_0)&\subset&\intr{D^n}-g^{-1}(\bar V)\\
          &=&g^{-1}(e)-g^{-1}(\bar V)\\
           &=&g^{-1}(e-\bar V)  \\
					&\subset& g^{-1}(e-V) 
					\end{eqnarray*}
i.e. $\intr{D^n}-B(0,r_0)\subset g^{-1}(e-V)$, therefore $g(\intr{D^n}-B(0,r_0))\subset e-V$.
	$$ S^{n-1} \subset D^n-B(0,r_0) = \overline{\intr{D^n}-B(0,r_0)} .$$ Therefore
\begin{eqnarray*}
f(S^{n-1}) &\subset &  f(\overline{\intr{D^n}-B(0,r_0)}) \\ 
          &\subset & \overline{f(\intr{D^n}-B(0,r_0)}) \\
           &=& \overline{g(\intr{D^n}-B(0,r_0)}) \\
					&\subset&\overline{e-V}\\
					&=&\overline{ e\cap (X-V)}\\
					&=&\overline{e} \cap \overline{(X-V)}\\
					&=&\overline{e} \cap (X-V)\\
					&= &\bar{e}-V.
\end{eqnarray*}
$y_0\notin\bar{e}-V$, thus $y_0\notin f(S^{n-1})$.
This shows that 
  
 $$f(\intr{D^n})\cap f(S^{n-1})=e\cap f(S^{n-1})=\emptyset$$
  
\item 
\begin{eqnarray*}
\partial e&=&\bar{e}-e \\
          &=&f(D^n)-f(\intr{D^n}) \\
           &=&(f(\intr{D^n})\cup f(S^{n-1}))-f(\intr{D^n})  \\
					&=& f(S^{n-1})   \qquad (\mbox{by } (2))
					\end{eqnarray*}
\end{enumerate}					
\end{proof}
%%%%%%%%%%%%%%%%%%%%%%%%%%%%%%%%%%%%%%%%%%%%%%%%%%%%%%%%%%%%%%%%%%%%%%%%%%%%%%%%%%%%%%%%%%%%%%%%%%%%%%%%%%%%%%%%
%%%%%%%%%%%%%%%%%%%%%%%%%%%%%%%%%%%%%%%%%%%%%%%%%%%%
\begin{lem}\label{l17}
Let $K$ be compact space, $X$ a Hausdorff space and $f:K\longrightarrow X$ continuous. Then the induced map
$\tilde{f}:K\longrightarrow f(K)$ is a quotient map.
\end{lem}
\begin{proof}
$\tilde{f}$ is continuous. Let $C$ be a subset of $f(K)$ such that $\tilde{f}^{-1}(C)$ is closed, and is therefore $\tilde{f}^{-1}(C)$ is compact. It follows that
$C=\tilde{f}(\tilde{f}^{-1}(C))$ is compact in $f(K)$ which is Hausdorff. Therefore $C$ is closed in $f(K)$. It follows that $\tilde{f}$ is a quotient map.
\end{proof}
\begin{lem}\label{l18}
Let $X \stackrel{q}{\longrightarrow}Y$ be a quotient map and $O$ an open subset of $Y$. Then the induced map
$\tilde{q}: q^{-1}(O)\longrightarrow O$ is a quotient map.
\end{lem}
\begin{proof}
Clear.
\end{proof}
\begin{lem}\label{l19}
Let $X$ be a Hausdorff space, $f:D^n\longrightarrow X$ a continuous map such that
\begin{enumerate}
	\item $f_{|\intr{D^n}}$ is injective.
	\item $f(\intr{D^n})\cap f(S^{n-1})=\emptyset$.
\end{enumerate}
Then $f(\intr{D^n})$ is an open cell in $X$ with characteristic map $f$.
\end{lem}
\begin{proof}
By Lemma \ref{l17} the induced map $D^n\longrightarrow f(D^n)$ is a quotient map, $f(\intr{D^n})\cap f(S^{n-1})=\emptyset$. Thus
$f(\intr{D^n})=f({D^n})-f(S^{n-1})$. It follows that $f(\intr{D^n})$ is open in $f(D^n)$.
By Lemma \ref{l18}, the induced map $\intr{D^n}\longrightarrow f(\intr{D^n})$
is a quotient map which is injective. Therefore the induced map $\intr{D^n}\longrightarrow f(\intr{D^n})$ is a homeomorphism.
It follows that $f(\intr{D^n})$ is an open cell in $X$ with characteristic map $f$.
\end{proof}
\begin{lem}\label{l20}
Let $X,Y$ be two Hausdorff spaces, $f:X\longrightarrow Y$ a continuous map and $e$ an open cell in $X$. Assume that
\begin{enumerate}
	\item $f_{|e}$ is injective.
	\item $f(e)\cap f(\partial e)=\emptyset$.
\end{enumerate}
Then $f(e)$ is an open cell in $Y$, $\dim f(e)=\dim e$, $\overline{f(e)}=f(\bar{e})$
and $\partial f(e)=f(\partial e)$.
\end{lem}
\begin{proof}
Immediate consequence of Lemma \ref{l19} and Lemma \ref{l16}.
\end{proof}
\begin{definition}\label{d3}
A CW-complex is a Hausdorff space $X$ together with a partition of  $X$ into open cells such that
\begin{enumerate}
\item For each open cell $e$ of of the partition of $X$, $\partial e$ is contained in the union of finitely many open cells (of the partition of $X$) of lower dimensions.
	\item A subset $C$ of $X$ such that $\overline{e}\cap C$ is closed for every open cell $e$ of the partition of $X$ is itself closed.  
	\end{enumerate}
\end{definition}
\begin{lem}\label{l22}
Let $X$ be a Hausdorff space space such that 
\begin{enumerate}
	\item $X$ has a finite partition into open cells.
	\item If $e$ is an open cell in the partition of $X$ and $x\in\partial e$, then the unique open cell $e^{'}$ containing $x$ is of lower dimension than $e$.
	\end{enumerate}
Then $X$ is a CW-complex.
\end{lem}
\begin{proof}
Immediate consequence of the previous definition.
\end{proof}
A CW-complex with finitely many open cells is said to be finite;  a CW-complex is finite iff it is compact.
	\begin{lem}\label{l21}
	Let $X_1,X_2,\dots,X_q$ be Hausdorff topological spaces and $e_1,e_2,\dots,e_q$ open cells in $X_1,X_2,\dots,X_q$ respectively,
	then $e_1\times e_2\times\dots e_q$ is an open cell in $X_1\times X_2\dots\times X_q$ with boundary
	$$\partial(e_1\times e_2\times\dots e_q)=\bigcup_{i=1}^q \bar{e}_1\times \bar{e}_2\times\dots\times\bar{e}_{i-1}\times \partial e_i\times\bar{e}_{i+1}\times\dots\times\bar{e}_q$$ and $$\dim e_1\times e_2\times\dots \times e_q= \dim e_1+\dim e_2+\dots +\dim e_q.$$
	\end{lem}
	\begin{proof}
By induction and is straightforward.
\end{proof}
\section {Stiefel and Grassmann manifolds}\label{s3}
Throughout this paper, $K$ is assumed to be either the field $\R$ of real numbers, the field $\C$ of complex numbers or the skew field $\H$ of quaternions. All vector spaces in this paper are vector spaces over $K$.\\
An inner product on a vector space $H$ is a map
$\langle\cdot,\cdot\rangle$
$$\begin{array}{rrll}
\langle\cdot,\cdot\rangle:&H\times H&\to&K\\
      &(x,y)&\mapsto&\langle x,y\rangle
			\end{array}$$
that satisfies the following axioms:\\ For all $x,y,z\in H$ and $\lambda\in K$
\begin{enumerate}
	\item linearity in the first argument:
	\begin{eqnarray*}
	\langle x+y,z\rangle&=&\langle x,z\rangle+\langle y,z\rangle\\
	\langle \lambda  x,y\rangle&=&\lambda \langle x,y\rangle 
	\end{eqnarray*}
	\item Conjugate symmetry:
	           $$\langle y,x\rangle =\overline{\langle x,y\rangle }$$
\item Positive definite:
$$\langle x,x\rangle \geq 0\mbox{ and }	\langle x,x\rangle =0\Rightarrow x=0$$
\end{enumerate}
					
A vector space on which there is defined an inner product is called an inner product vector space. An inner product vector space $H$ has a norm $\norm{\:}$ associated to its inner product given by
$\norm{x}=\langle x,x\rangle ^{\frac{1}{2}}$, $x\in H$. Given two finite dimensional inner product vector spaces $M$ and $N$, then any linear map $M\toby{u}N$ is continuous and the vector space  $L(M,N)$ of linear maps has a norm $\norm{\:}$ defined by $$\norm{u}=\sup_{\norm{x}\leq 1}\norm{u(x)}$$
A linear map $M\toby{u}N$ is said to be inner product preserving if
$$\langle u(x),u(y)\rangle = \langle x,y\rangle \mbox{ for all } x,y\in F.$$  
Such a map is a unit vector in the normed vector space $L(M,N)$.  
Define the Stiefel manifold $V(M,N)$ to be the topological subspace of $L(M,N)$ of inner product preserving linear maps from $M$ to $N$. $V(M,N)$ is a closed subset of the unit sphere of the finite dimensional normed vector space $L(M,N)$. It is therefore a compact metric space with distance $d$ given by 
$$d(u.v)=\norm{u-v}$$
Observe that if $P$ is another inner product vector space, then the composition induces a continuous map
$$V(N,P)\times V(M,N)\to V(M,P)$$
so that the finite dimensional inner product vector spaces form a topological category $V$, i.e. a category enriched over the cartesian monoidal category of topological spaces. Define an equivalence relation $\sim$ on the topological space $V(M,N)$ by 
$u\sim v$ if $u$ and $v$ have the same image.\\
Let $E$ be the vector space with basis $(e_n)_{n\in\N^*}$. A vector $x\in E$ may be written as 
$x=\sum_{n=1}^{\infty}x_ne_n$ with $x_n=0$ for $n$ sufficiently large. Define an inner product on $E$ by $$\langle x,y\rangle=\sum_{n=1}^{\infty}x_n\bar{y_n}.$$
Let $E_0$ be the zero subspace of $E$ and $E_n$ the inner product $K$-subspace of $E$ generated by $(e_k)_{1\leq k\leq n}$, $n\in\N^*$. For $n\in \N$ and $H\in V$, let $V(n,H)=V(E_n,H)$ and define $G(n,H)$ to be the quotient space of 
$V(n,H)$ by the equivalence relation $\sim$ defined above. $V(n,H)$ and $G(n,H)$ are the ordinary Stiefel manifold and Grassmann manifold over $K$.
Let
\begin{equation}\label{eq1} 
p(n,H):V(n,H)\rightarrow G(n,H) 
\end{equation} 
be the quotient map. Observe that any subspace of $H$ having dimension $n$ is the image of certain inner product preserving linear map from $E_n$ to $H$ and any two such maps are equivalent.
We may therefore view $G(n,H)$ as the set of $n$-dimensional subspaces of $H$ and $p(n,H)$ as the map that takes a linear map $u\in V(n,H)$ to its image and think of $G(n,H)$ as having the final topology induced by the map $p(n,H)$.
 For $n, m  \in \N$, define
$V(n,m)=V(n,E_m)$, $G(n,m)=G(n,E_m)$ and $p(n,m)=p(n,E_m)$.
The next proposition is a classical result which gives a sufficient condition for the quotient space to be Hausdorff \cite[Proposition 1.6 page 140]{D}.  
\begin{pro}\label{p1}
 Let $\mr$ be an equivalence relation on a topological space $X$ and assume that
\begin{enumerate}
	\item The quotient map $p:X\to X/\mr$ is open .
	\item The graph $\Gamma$ of the equivalence relation $\mr$ is closed in $X\times X$.
\end{enumerate}
Then the quotient space $X/\mr$ is Hausdorff.
\end{pro}
\begin{pro}\label{p2} The space $G(n,H)$ is compact Hausdorff,  $n \in \N$ and $H\in V$.
\end{pro}
\begin{proof}
$G(n,H)$ is the image by the continuous map $p(n,H)$ of the compact space $V(n,H)$, thus $G(n,H)$ is compact. Let $O$ be open in $V(n,H)$.
$$p(n,H)^{-1}(p(n,H)(O))=\bigcup_{u\in V(n,n)}Ou$$
is open, therefore the quotient map $p(n,H)$ is open.
Let $\Gamma$ be the graph of the equivalence relation $\sim$.
$\Gamma$ is the image of the following continuous map
$$\begin{array}{ccc}
V(n,H)\times V(n,n)& \longrightarrow &V(n,H)\times V(n,H)\\
                  (v,u)& \longmapsto & (v,vu)
		\end{array}	
	$$
	$V(n,H)\times V(n,n)$ is compact, therefore $\Gamma$ is a compact subset of the Hausdorff space $V(n,H)\times V(n,H)$. It follows that $\Gamma$ is closed. By Proposition \ref{p1}, $G(n,H)$ is Hausdorff.
\end{proof}
\section{The cell structure for Grassmann manifolds}\label{s4}
In this section, we show that the method of Milnor and Stasheff for the construction of the cell structure for real Grassman manifolds \cite[Chapter 6]{MS} extends easily to the complex and quaternionic cases.
	
Define an elementary Schubert symbol of type $(n,m)$ to be a strictly increasing function $\{1,2,\dots,n\}\toby{\sigma}\{1,2,\dots,m\}$. The set of elementary Schubert symbols of type $(n,m)$ is denoted by $S(n,m)$.
There is an order relation on $S(n,m)$ given by
\begin{equation}\label{eq8}
\sigma\leq \tau  \iff  \sigma(k)\leq \tau(k), 1\leq k\leq n\end{equation}
Define a dimension function
$$\begin{array}{lrrl}
			d:&S(n,m)&\to&\N\\
			     &\sigma&\mapsto&d(\sigma)=\sum_{k=1}^n\sigma(k)-k
		\end{array}$$			
Let $\sigma,\tau\in S(n,m)$ with $\sigma <\tau$, then $d(\sigma)< d(\tau)$ so that $d$
is a strictly increasing function.

%%%%%%%%%%%%%%%%%%%%%%%%%%%%%%%%%%%%%%%%%%%%%%%%%%%%%%%%%%%%%%%%%%%%%%%%%%%%%
%%%%%%%%%%%%%%%%%%%%%%%%%%%%%%%%%%%%%%%%%%%%%%%%%%%%%%%%%%%%%%%%%%%%%%%%%%%%%
%%%%%%%%%%%%%%%%%%%%%%%%%%%%%%%%%%%%%%%%%%%%%%%%%%%%%%%%%%%%%%%%%%%%%%%%%%%%%
Let $H$ be finite dimentional inner product vector space of dimension $m$, $(h_k)_{1\leq k\leq m}$ an orthonormal basis of $H$, $H_k$ the subspace of $H$ generated by $(h_i)_{1\leq i\leq k}$, $1\leq k\leq m$. Define
\begin{equation}\label{eq2} 
V_{\sigma}(n,H)=\{u\in V(n,H)\suchthat u(E_k)\subset H_{\sigma(k)}\mbox{ and } \langle u(e_k) , h_{\sigma(k)}\rangle >0, 1\leq k\leq n\}
\end{equation}
and $V_{\sigma}(n,m)=V_{\sigma}(n,E_m)$.
The closure $\overline{V_{\sigma}(n,H)}$ is given by
\begin{equation}\label{eq3} 
\overline{V_{\sigma}(n,H)}=\{u\in V(n,H)\suchthat u(E_k)\subset H_{\sigma(k)} \mbox{ and }	\langle u(e_k) , h_{\sigma(k)}\rangle\geq 0, 1\leq k\leq n\}
\end{equation}
%%%%%%%%%%%%%%%%%%%%%%%%%%%%%%%%%%%%%%%%%%%%%%%%%%%%%%%%%%%%%%%%%%%%%%%%%%%%%%%%%%%%%%%%%%%%%%%%%%%%%%%%%%%%%%%%%%%%%%%%%%%%%%%%%%%%%%%%%%
%%%%%%%%%%%%%%%%%%%%%%%%%%%%%%%%%%%%%%%%%%%%%%%%%%%%%%%%%%%%%%%%%%%%%%%%%%%%%%%%%%%%%%%%%%%%%%%%%%%%%%%%%%%%%%%%%%%%%%%%%%%%%%%%%%%%%%%%%%
%%%%%%%%%%%%%%%%%%%%%%%%%%%%%%%%%%%%%%%%%%%%%%%%%%%%%%%%%%%%%%%%%%%%%%%%%%%%%%%%%%%%%%%%%%%%%%%%%%%%%%%%%%%%%%%%%%%%%%%%%%%%%%%%%%%%%%%%%%
Let $\dim_{\R}K$ be the dimension of $K$ as an  $\R$-vector space. 
\begin{lem}\label{l1}
$V_{\sigma}(n,H)$ is an open cell in $V(n,H)$ of dimension $\dim_{\R}K.d(\sigma)$, $\sigma\in S(n,m)$.
\end{lem}
\begin{proof}
First, observe that the inner product preserving map $f:E_m\longrightarrow H$ given $f(e_k)=h_k$ is a isomorphism in the category $V$, therefore it induces a homeomorphism $V(n,m)\longrightarrow V(n,H)$ which maps $V_{\sigma}(n,m)$ onto $V_{\sigma}(n,H)$, we may then assume without loss of generalities that $H=E_m$. 
In  \cite[Lemma 6.2.]{MS}, it is shown that for $K=\R$, $V_{\sigma}(n,m)$ is an open cell in $V(n,m)$ of dimension $d(\sigma)$. The method used there extends readily to the other cases of $K$. The only minor modification is due to the fact that the inner product is not symmetric when $K=\C$ or $K=\H$ and the formula of the rotation $T(u,v)$ defined in \cite [page 77]{MS} should be 
$$T(u,v)(x)=x-\frac{\langle x,u+v\rangle}{1+\langle u,v\rangle}(u+v)+2\langle x,u\rangle v.$$
where $u$ and $v$ are two unit vectors in $E_m$ with $\langle u,v\rangle>0$; ($\langle u,v\rangle $ is a positive real number).
\end{proof}	
\begin{definition}\label{d1}
Let $W$ be a finite dimensional vector space of dimension $m$.
\begin{enumerate}
	\item 
A  flag in  $W$ is a strictly increasing sequence
 $0=W_0\subset W_1\subset W_2\subset\dots\subset W_q=W$ of subspaces of $W$. If we write $d_k = \dim W_k$, then $0=d_0 \prec d_1\prec d_2\prec \dots \prec d_q=m$. The sequence $(d_0,d_1,\dots, d_q)$ is called the signature of the flag $(W_k)_{0\leq k\leq q}$.
\item A flag $(W_k)_{0\leq k\leq q}$ in $W$ is said to be complete if $q=m$, in which case $\dim W_k = k$, $0\leq k\leq m$ and the flag has signature $(0,1,\dots,m)$.
\item A vector space that is assigned a fixed flag (resp. a fixed complete flag) is said to be flagged (resp. completely flagged).
 \end{enumerate}  
\end{definition}
Let $W$ be a completely flagged finite dimensional vector space with dimension $m$. $(W_k)_{0\leq k\leq m}$ its fixed complete flag and $X$ an $n$-dimensional vector subspace of $W$. Define the Schubert function of $X$ to be the function
\begin{equation}\label{eq10} 
\begin{array}{rrcl}
\omega_X^W:&\{0,1,\dots,m\}&\to&\{0,1,\dots,n\}\\
    &k&\mapsto&\dim X\cap W_k
\end{array}
\end{equation}
Then
\begin{itemize}
\item $0=\omega_X^W(0)\leq \omega_X^W(1)\leq\dots \leq \omega_X^W(m)=n$, so that $\omega_X^W$ is increasing.
	\item $ \omega_X^W(k+1)\leq \omega_X^W(k)+1, \quad 0\leq k\leq m-1$.
	\item $\omega_X^W$is surjective.
\end{itemize}
Define the Schubert symbol of $X$ to be the elementary Schubert symbol
\begin{equation}\label{eq9} 
\begin{array}{rrcl}
\sigma_X^W:&\{1,2,\dots,n\}&\to&\{1,2,\dots,m\}\\
    &k&\mapsto&\min (\omega_X^W)^{-1}(k)
\end{array}
\end{equation}
	$\sigma_X^W$ is characterized by the following property
\begin{equation}\label{eq13}  \dim X\cap W_{i}\geq k  \iff i\geq \sigma_X^W(k),\quad 1\leq k\leq n.
\end{equation}
In particular, if we set $X_0=0$ and $X_k=X\cap W_{\sigma_X^W(k)}$, $1\leq k\leq n$,  then
\begin{equation}\label{eq11} 
0=X_0\subset X_1\subset\dots\subset X_n=X
\end{equation} 
is a complete flag in $X$ called the induced complete flag, so that a subspace of a completely flagged finite dimensional vector space is naturally a completely flagged vector space.
\begin{lem}\label{l3}
Let $Z$ be a finite dimensional completely flagged vector space, $Y$ a subspace of $Z$ and $X$ a nonzero subspace of $Y$. Then with the above notation, we have 
\begin{enumerate}
\item  $\omega_X^Z = \omega_X^Y \omega_Y^Z.$
\item  $\sigma_X^Z = \sigma_Y^Z \sigma_X^Y.$
\end{enumerate}
\end{lem}	
\begin{proof}
Let $p=\dim X$, $q=\dim Y$ and $r=\dim Z$.
\begin{enumerate}
\item For $k\in \{1,2,\dots,r\}$
\begin{eqnarray*}
 \omega_X^Y \omega_Y^Z(k)&=&\omega_X^Y(\dim Y\cap Z_k)\\
 &=&\dim X\cap Y\cap Z_k\\
 &=&\dim X\cap Z_k\\
 &=&\omega_X^Z(k) 
\end{eqnarray*}
We then have $\omega_X^Z = \omega_X^Y \omega_Y^Z.$
\item For $k\in \{1,2,\dots,p\}$
\begin{eqnarray*}
 \sigma_X^Z(k)&=&\min  (\omega_X^{Z })^{-1}(k) \\
 &=&\min(\omega_X^Y \omega_Y^Z)^{-1}(k)\\
 &=&\min (\omega_Y^{Z })^{-1}((\omega_X^{Y })^{-1}(k)) \\
 &=&\min (\omega_Y^{Z}) ^{-1}(\min (\omega_X^{Y })^{-1}(k))\\
 &=&\sigma_Y^Z \sigma_X^Y(k)
\end{eqnarray*}
We then have $\sigma_X^Z = \sigma_Y^Z \sigma_X^Y.$   
\end{enumerate}
\end{proof}
Let $H$ be a finite dimensional vector space with orthonormal basis $(h_k)_{1\leq k\leq m}$, $H_0=0$ and $H_k$ the subspace of $H$ generated by $(h_i)_{1\leq i\leq k}$, $1\leq k\leq m$. $(H_k)_{0\leq k\leq m}$ is a complete flag in $H$. Define 
$$G_{\sigma}(n,H)= \{X\in G(n,H)\suchthat\sigma_X^H=\sigma\},\quad \sigma \in S(n,m)$$
The following lemma is clearly an immediate consequence of this last definition. 
\begin{lem}\label{l2}The family  $(G_ {\sigma}(n,H))_{\sigma \in S(n,m)}$ is a partition of $G(n,H)$.
 \end{lem}
\begin{lem}\label{l4}
Let $\sigma\in S(n,m)$, then
\begin{enumerate}
\item $p(n,H)$ induces a bijection $p(n,H)_{/}:V_{\sigma}(n,H) \longrightarrow G_{\sigma}(n,H)$. 
\item Let $v\in \partial V_{\sigma}(n,H)$ and $Y=p(n,H)(v)$, then $\sigma_Y\prec\sigma$ where $\sigma_Y$ is the Schubert symbol of $Y$. In particular 
 $$p(n,H)(V_{\sigma}(n,H))\cap p(n,H)(\partial V_{\sigma}(n,H))=\emptyset .$$
\end{enumerate}
\end{lem}
\begin{proof}
\begin{enumerate}
\item 
\begin{itemize}
\item Let $u\in V_{\sigma}(n,H)$, $X=p(n,H)(u)$ and $k\succ 0.$\\ 
$ X\cap  H_{\sigma(k)} \subset \Im u$ and $u$ is injective. Thus
\begin{eqnarray*}
 \dim X\cap  H_{\sigma(k)}&=&\dim u^{-1}(X\cap  H_{\sigma(k)}) \\
 &=&\dim u^{-1}(H_{\sigma(k)})\\
 &=&\dim  E_k\\
 &=&  k
\end{eqnarray*}
 while
\begin{eqnarray*}
 \dim X\cap  H_{\sigma(k)-1}&=&\dim u^{-1}(X\cap  H_{\sigma(k)-1}) \\
 &=&\dim u^{-1}( H_{\sigma(k)-1})\\
 &=&\dim  E_{k-1}\\
 &=& k-1.
\end{eqnarray*} 
It follows that $\sigma$ is the Schubert symbol of $X$ and $$p(n,H)(V_{\sigma}(n,H)) \subset G_{\sigma}(n,H).$$
\item 
Let $X \in G_{\sigma}(n,H)$ and assume that $u \in V_{\sigma}(n,H)$ is such that\\ $p(n,H)(u)=X$. Let $k\in \{1,2,\dots,n\}$ 

$$\left\{
\begin{array}{l}
 u(E_{k})=X\cap H_{\sigma(k)}\\
u(E_{k-1})=X\cap H_{\sigma(k)-1}
\end{array}
\right.
$$
\begin{eqnarray*}
u(e_{k})&\in &u(E_{k}) \cap u(E_{k-1}^{\bot}) \\
 &=&(X\cap H_{\sigma(k)})\cap (X\cap H_{\sigma(k)-1})^{\bot} 
\end{eqnarray*} 
It follows that $u(e_{k})$ is the unique unit vector $x_{k}$ in $(X\cap H_{\sigma(k)})\cap(X\cap H_{\sigma (k)-1})^{\bot}$ such that $\langle x_{k}, h_{\sigma(k)}\rangle \succ 0$. This shows that the map $$p(n,H)_{/}:V_{\sigma}(n,H) \longrightarrow G_{\sigma}(n,H)$$ is injective. To show that it is surjective, we only need to construct an antecedent $u$ for a given element $X \in G_{\sigma}(n,H)$ as suggested by the previous argument.  

\end{itemize}
\item
\begin{itemize}
\item Recall that 
$$V_{\sigma}(n,H)=\{u\in V(n,H)\suchthat u(E_k)\subset H_{\sigma(k)}\mbox{ and } \langle u(e_k),h_{\sigma(k)}\rangle >0, 1\leq k\leq m\}$$
$$\overline{V_{\sigma}(n,H)}=\{u\in V(n,H)\suchthat u(E_k)\subset H_{\sigma(k)}\mbox{ and } \langle u(e_k),h_{\sigma(k)}\rangle\geq 0, 1\leq k\leq n\}$$
$v\in \partial V_{\sigma}(n,H)=\overline{V_{\sigma}(n,H)}-V_{\sigma}(n,H)$. Therefore
$$\left\{
\begin{array}{l}
 v(E_k)\subset H_{\sigma(k)}, 1\leq k\leq n\\
\exists k_0\in\{1,2,\dots,n\} \mbox{ such that }  v(e_{k_0})\in H_{\sigma(k_0)-1}
\end{array}
\right.
$$
It follows that
$$\left\{
\begin{array}{l}
 v(E_k)\subset Y\cap H_{\sigma(k)}, 1\leq k\leq n\\
v(E_{k_0})\subset Y\cap H_{\sigma(k_0)-1}
\end{array}
\right.
$$
and
$$\left\{
\begin{array}{l}
 \dim Y\cap  H_{\sigma(k)}\geq k, 1\leq k\leq n\\
\dim Y\cap  H_{\sigma(k_0)-1}\geq k_0
\end{array}
\right.
$$
by (\ref{eq13})
$$\left\{
\begin{array}{l}
 \sigma_Y(k)\leq \sigma(k), 1\leq k\leq n\\
\sigma_Y(k_0) < \sigma(k_0)
\end{array}
\right.
$$
i.e. $\sigma_Y < \sigma$          
\item
If $X\in p(n,H)(V_{\sigma}(n,H))$, then $\sigma_X = \sigma >\sigma_Y$. Therefore $X\neq Y$ and $p(n,H)(V_{\sigma}(n,H))\cap p(n,H)(\partial V_{\sigma}(n,H))=\emptyset$.
\end{itemize} 
\end{enumerate}
\end{proof}
\begin{lem}\label{l5}
Let $\sigma\in S(n,m)$, then
\begin{enumerate}
\item $G_{\sigma}(n,H)$ is an open cell in $G(n,H)$ with dimension $\dim_{\R}K.d(\sigma)$. 
\item $\partial G_{\sigma}(n,H)=p(n,H)(\partial V_{\sigma}(n,H))$.
\item Let $Y\in \partial G_{\sigma}(n,H)$ and $\sigma_Y$ the Schubert symbol of $Y$, then $\sigma_Y\prec\sigma$. In particular 
$d(\sigma_Y) \prec d(\sigma)$. 
\end{enumerate}
 \end{lem}	
\begin{proof}
Immediate consequence of Lemma \ref{l20}, Lemma \ref{l1} and lemma \ref{l4}.
\end{proof}
%
%
%
%
%
%
%
%%%%%%%%%%%%%%%%%%%%%%%%%%%%%%%%%%%%%%%%%%%%%%%%%
%%%%%%%%%%%%%%%%%%%%%%%%%%%%%%%%%%%%%%%%%%%%%%%%%%%%%%%%%%%%%%%%%%%%%%%%%%%
To sumarize, one has the following 
\begin{thm}\label{t1}
$G(n,H)$ has a finite CW-complex structure with open cells
$G_{\sigma}(n,H)$, $\sigma\in S(n,m).$
\end{thm}
\begin{proof}
Immediate consequence of Lemma \ref{l22}, Lemma \ref{l2} and lemma \ref{l5}.

\end{proof}
%%%%%%%%%%%%%%%%%%%%%%%%%%%%%%%%%%%%%%%%%%%%%%%%%%%%%%%%%%%%%%%%%%%%%%%%%%%%
%%%%%%%%%%%%%%%%%%%%%%%%%%%%%%%%%%%%%%%%%%%%%%%%%%%%%%%%%%%%%%%%%%%%%%%%%%%%
\begin{cor}\label{c1}
$G(n,H)$ is a topological manifold of dimension $n(m-n).\dim _{\R}K$.	
\end{cor}
The proof is similar but not identical to the one given in \cite [page 31] {H2}.
\begin{proof}
$G(n,H)$ is Hausdorff. Let $X_0\in G(n,H)$ and choose $(h_1,h_2,\dots,h_m)$ to be a basis of $H$ such that
$(h_{m-n+1},\dots,h_{m-1},h_m)$ is a basis of $X_0$. Define $u_0\in V(n,H)$ by 
$u_0(e_k)=h_{m-n+k}, 1\leq k\leq n$. $p(n,H)(u_0)=X_0$, $u_0\in V_{\sigma}(n,H)$ where $\sigma \in S(n,m)$ is given by $\sigma(k)=m-n+k$. By Lemma \ref{l4}, $X_0\in G_{\sigma}(n,H)$.
$G_{\sigma}(n,H)$ is an open cell in $G(n,H)$ of top dimension, thus $G_{\sigma}(n,H)$ is an open neighborhood of $X_0$.
\begin{eqnarray*}
\dim G_{\sigma}(n,H)&=&\dim_{\R}K.d(\sigma)\\
                           &=&\dim_{\R}K.n(m-n).
\end{eqnarray*}
It follows that $G(n,H)$ is a topological manifold of 
dimension $\dim_{\R}K.n(m-n)$.
\end{proof}		
%%%%%%%%%%%%%%%%%%%%%%%%%%%%%%%%%%%%%%%%%%%%%%%%%%%%%%%%%%%%%%%%%%%%%%%%%%%%%%%%%%%%%%%%%%%%%%%%%%%%%%%%%%%%%%%%%%%%%%%%%%%%%%%%%%%%%%%%%%%%%%%%%%%%%%%%%%%%%%%%%%%%%%%%%%%%%%%%%%%%%%%%%%%%%%%%%%%%%%%%%%%%%%%%%%%%%%%%%%%%%%%%%%%%%%%%%%%%%%%%%%%%%%%%%%%%%%%%%%%%%%%%%%%%%%%%%%%%
%%%%%%%%%%%%%%%%%%%%%%%%%%%%%%%%%%%%%%%%%%%%%%%%%%%%%%%%%%%%%%%%%%%%%%%%%%%%%%%%%%%%%%%%%%%%%%%%%%%%%%%%%%%%%%%%%%%%%%%%%%%%%%%%%%%%%%%%%%%%%%%%%%%%%%%%%%%%%%%%%%%%%%%%%%%%%%%%%%%%%%%%%%%%%%%%%%%%%%%%%%%%%%%
\section{ Schubert cells in a generalized Stiefel manifold}\label{s5}
Let $H$ be a finite dimensional inner product vector space of dimension $n$. $(n_k)_{1\leq k\leq q}$ a sequence of integers such that $1 \leq n_1 \prec n_2 \dots \prec n_{q} \prec n$.
Define the generalized Stiefel manifold $V(n_1,\dots,n_q,H)$ to be the subspace of $\prod^{q}_{k=1}V(n_k,H)$ given by
\begin{equation}\label{eq4} 
V(n_1,\dots,n_q,H)=
\{{(v_k)_{1\leq k\leq q }}\in \prod^{q}_{k=1}V(n_k,H) \suchthat
\Im v_1\subset \Im v_2\subset\dots\subset \Im v_{q}\}
\end{equation}
$V(n_1,\dots,n_q,H)$ is a subspace of a Haudorff space, it is then Hausdorff. Let
$$\begin{array}{rccc}
\psi(n_1,\dots,n_q,H):&\prod^{q-1}_{k=1}V(n_k,n_{k+1})\times V(n_q,H) & \longrightarrow & \prod^{q}_{k=1}V(n_k,H)\\
               &(u_k)_{1\leq k\leq q} &\mapsto & (u_{q}\dots  u_{k+1} u_{k})_{1\leq k\leq q} 
\end{array}$$
$\psi(n_1,\dots,n_q,H)$ is continuous injective map with image $V(n_1,\dots,n_q,H)$. \\$\prod^{q-1}_{k=1}V(n_k,n_{k+1})\times V(n_q,H)$ is compact, therefore $V(n_1,\dots,n_q,H)$ is compact Hausdorff and
$\psi(n_1,\dots,n_q,H)$ induces a homeomorphism of\\ $\prod^{q-1}_{k=1}V(n_k,n_{k+1})\times V(n_q,H)$ onto $V(n_1,\dots,n_q,H)$.\\

Define
\begin{equation}\label{eq19} 
S(n_1,n_2,\dots,n_q,n)=S(n_1,n_2)\times S(n_2,n_3)\times\dots\times S(n_{q},n)
\end{equation}
An element of $S(n_1,n_2,\dots,n_q,n)$ is called a general Schubert symbol of type
$(n_1,n_2,\dots,n_q,n)$. It is then a finite sequence $\sigma=(\sigma_1,\sigma_2,\dots,\sigma_{q})$
of elementary Schubert symbols.
The order relation $\leq$ on the elementary Schubert symbols defined by (\ref{eq8}) induces an order relation $\leq$ on $S(n_1,n_2,\dots,n_q,n)$ given by
\begin{equation}\label{eq14}  
\sigma=(\sigma_1,\sigma_2,\dots,\sigma_{q})\leq\tau=(\tau_1,\tau_2\dots,\tau_{q}) \iff  \sigma_i\leq\tau_i, 1\leq i\leq q
\end{equation} 
Let 
$$\begin{array}{rccl}
d:&S(n_1,n_2,\dots,n_q,n)&\to&\N\\
    &\sigma=(\sigma_1,\sigma_2,\dots,\sigma_{q})&\mapsto&d(\sigma)=\sum_{i=1}^{q}d(\sigma_i)
		\end{array}$$
then
$$d(\sigma)=\sum_{i=1}^{q}d(\sigma_i)=\sum^{q}_{i=1}\sum^{n_i}_{k=1}\sigma_i(k)-k$$
If $\sigma <\tau$, then $d(\sigma)<d(\tau)$ so that $d$ is a strictly increasing function.\\
\begin{lem}\label{l23}
Let $\sigma=(\sigma_1,\sigma_2,\dots,\sigma_q)\in S(n_1,n_2,\dots,n_q,n)$ then
\begin{eqnarray*}
d(\sigma)&=&\sum_{k=1}^{n_1}\sigma_q\dots\sigma_2\sigma_1(k)-k +\sum_{k\notin Im \sigma_1}\sigma_q\dots\sigma_2(k)-k +\dots \\
            &&+ \sum_{k\notin Im\sigma_{q-2}}\sigma_q\sigma_{q-1}(k)-k +\sum_{k\notin Im\sigma_{q-1}}\sigma_q(k)-k.\\
	\end{eqnarray*}
\end{lem}
\begin{proof}
By Induction, for $q=2$
\begin{eqnarray*}
d(\sigma)&=&\sum_{k=1}^{n_1}\sigma_1(k)-k+\sum_{k=1}^{n_2}\sigma_2(k)-k\\
         &=&\sum_{k=1}^{n_1}\sigma_1(k)-k+\sum_{k=1}^{n_1}\sigma_2\sigma_1(k)-\sigma_1(k) 
				  +\sum_{k\notin Im \sigma_1}\sigma_2(k)-k\\
				&=&	\sum_{k=1}^{n_1}\sigma_2\sigma_1(k)-k	+ \sum_{k\notin Im \sigma_1}\sigma_2(k)-k.
\end{eqnarray*}
A similar argument shows that the formula is true for $q+1$ if it holds for $q$.
\end{proof}
Let $(h_i)_{1\leq i\leq n}$ an orthonormal basis of $H$.
For $ \sigma=(\sigma_1,\sigma_2,\dots,\sigma_q) \in S(n_1,n_2,\dots,n_q,n)$, define $V_{\sigma}(n_1,\dots,n_q,H)$ to be the subset of  $V(n_1,\dots,n_q,H)$ given by
\begin{equation}\label{eq21}  
V_{\sigma}(n_1,\dots,n_q,H)=\psi(n_1,\dots,n_q,H)(\prod^{q-1}_{k=1} V_{\sigma_k}(n_k,n_{k+1})\times V_{\sigma_q}(n_q,H))
\end{equation}  
\begin{lem}\label{l7}
Let $ \sigma=(\sigma_1,\sigma_2,\dots,\sigma_q) \in S(\sigma_1,\sigma_2,\dots,\sigma_q,n)$, then
$V_{\sigma}(n_1,\dots,n_q,H)$ is an open cell in $V(n_1,\dots,n_q,H)$ of dimension $\dim_{\R}K.d(\sigma)$ and boundary\\ $\partial V_{\sigma}(n_1,\dots,n_q,H)$ is equal to
 $$\bigcup^{q}_{k=1} \psi(n_1,\dots,n_q,H) (\overline{V_{\sigma_1}(n_1,n_2)}\times \dots \times \partial V_{\sigma_k}(n_k,n_{k+1})\times \dots \times \overline{V_{\sigma_{q}}(n_{q},H)})$$
\end{lem}
\begin{proof}
Immediate consequence of Lemma \ref{l21} and the fact that
$$\psi(n_1,\dots,n_q,H): \prod^{q-1}_{k=1} V(n_k,n_{k+1})\times V(n_q,H) \longrightarrow  V(n_1,\dots,n_q,H)$$ is a homeomorphism. 
\end{proof}
\begin{lem}\label{l8}
Let $m_{1}, m_{2},...,m_{r},m$ be positive integers with  $m_{1}\prec m_{2}\prec ...\prec m_{r}\prec m$, $F$ an $m$-dimensional inner product vector space, $(f_{i})_{1\leq i \leq m}$ an orthonormal basis of $F$, $ \sigma=(\sigma_1,\sigma_2,...,\sigma_r) \in S(m_1,m_2,...,m_{r},m)$ and
$$\phi(m_1,\dots,m_r,F): \prod^{r-1}_{k=1} V(m_k,m_{k+1})\times V(m_r,F) \longrightarrow  V(m_1,F)$$
the composition map. Then
\begin{enumerate}
\item $\phi(m_1,\dots,m_r,F)(\prod^{r-1}_{k=1} V_{\sigma_{k}}(m_k,m_{k+1})\times  V_{\sigma_{r}}(m_r,F)) \subset V_{\sigma_r \dots\sigma_2\sigma_1}(m_1,F).$
\item $\phi(m_1,\dots,m_r,F)(\partial V_{\sigma_{1}}(m_1,m_2)\times \prod^{r-1}_{k=2} V_{\sigma_{k}}(m_k,m_{k+1})\times  V_{\sigma_{r}}(m_r,F))\\ \subset  \partial V_{\sigma_r \dots\sigma_2\sigma_1}(m_1,F).$ 
\end{enumerate}
\end{lem}
\begin{proof}
The first assertion is proved by induction on $r$ and is straightforward. Similarly, the second assertion is proved by induction on $r$. It uses the first assertion and (\ref{eq3}).
\end{proof}
\begin{lem}\label{l9}
Let $\sigma=(\sigma_1,\sigma_2,\dots,\sigma_q)\in S(n_1,n_2,\dots,n_q,n)$, then
\begin{enumerate}
	\item $V_{\sigma}(n_1,n_2,\dots,n_q,H)=V(n_1,n_2,\dots,n_q,H)\cap  \prod^{q}_{k=1}
	V_{\sigma_q\dots\sigma_{k+1}\sigma_k}(n_k,H).$
	\item $\partial V_{\sigma}(n_1,n_2,\dots,n_q,H)\subset 
	\cup_{k=1}^q V(n_1,H)\times\dots\times V(n_{k-1},H)\times\\ 
	\partial V_{\sigma_q\dots\sigma_{k+1}\sigma_k}(n_k,H)\times 
	             V(n_{k+1},H)\times\dots\times V(n_q,H).$
\end{enumerate}
\end{lem}
\begin{proof}
%%%%%%%%%%%%%%%%%%%%%%%%%%%%%%%%%%%%%%%%%%%%%%%%%%%%%%%%%%%%%%%%%%%%%%%%%%%%%%%%%%%%%%%%%%%%%%%%%%%%%%%%%%%%%%%%%%%%%%%%%%%%%%%%%%%%%%%%%%
%%%%%%%%%%%%%%%%%%%%%%%%%%%%%%%%%%%%%%%%%%%%%%%%%%%%%%%%%%%%%%%%%%%%%%%%%%%%%%%%%%%%%%%%%%%%%%%%%%%%%%%%%%%%%%%%%%%%%%%%%%%%%%%%%%%%%%%%%%
%%%%%%%%%%%%%%%%%%%%%%%%%%%%%%%%%%%%%%%%%%%%%%%%%%%%%%%%%%%%%%%%%%%%%%%%%%%%%%%%%%%%%%%%%%%%%%%%%%%%%%%%%%%%%%%%%%%%%%%%%%%%%%%%%%%%%%%%%%
\begin{enumerate}
	\item
By Lemma \ref{l8}
$$V_{\sigma}(n_1,n_2,\dots,n_q,H) \subset V(n_1,n_2,\dots,n_q,H)\cap  \prod^{q}_{k=1}
	V_{\sigma_q\dots\sigma_{k+1}\sigma_k}(n_k,H)$$
	We prove by induction on $q$ that  
	$$  V(n_1,n_2,\dots,n_q,H)\cap  \prod^{q}_{k=1}
	V_{\sigma_q\dots\sigma_{k+1}\sigma_k}(n_k,H)\subset V_{\sigma}(n_1,n_2,\dots,n_q,H)$$
	For $q=2$, let $(v_1,v_2)\in V(n_1,n_2,H)\cap 
	V_{\sigma_2 \sigma_1}(n_1,H)\times V_{\sigma_2}(n_2,H)$, there exists $(u_1,u_2)\in V(n_1,n_2)\times V(n_2,H)$ such that
	$(v_1,v_2)= (u_2u_1,u_2)$. $u_2=v_2 \in V_{\sigma_2}(n_2,H)$. To show that $u_1 \in V_{\sigma_1}(n_1,n_2)$, let $k\in \{1,2,\dots,n_1\}$,
	$u_2u_1(E_k) = v_1(E_k)\subset H_{\sigma_2\sigma_1(k)}$, 
	thus $u_1(E_k)\subset u_2^{-1}(H_{\sigma_2\sigma_1(k)})\subset E_{\sigma_{1}(k)}.$
	Write $u_1(e_k)= \lambda e_{\sigma_1(k)}+x$ where $x\in E_{\sigma_1(k)-1}$.
\begin{eqnarray*}
\langle u_2u_1(e_k), e_{\sigma_2\sigma_1(k)}\rangle&=& \langle u_2(\lambda e_{\sigma_1(k)}+x),e_{\sigma_2\sigma_1(k)}\rangle\\
             &=&\langle \lambda u_2( e_{\sigma_1(k)}),e_{\sigma_2\sigma_1(k)}\rangle\\
						&=& \lambda \langle u_2( e_{\sigma_1(k)}),e_{\sigma_2\sigma_1(k)}\rangle 
\end{eqnarray*}
In the other hand
$$\langle u_2u_1(e_k), e_{\sigma_2\sigma_1(k)}\rangle = \langle v_1(e_k), e_{\sigma_2\sigma_1(k)}\rangle\succ 0$$
or $\langle u_2( e_{\sigma_1(k)}),e_{\sigma_2\sigma_1(k)}\rangle \succ 0$ therefore $\lambda\succ 0$, i.e. 
$\langle u_1(e_k), e_{\sigma_1(k)}\rangle\succ 0$, $1\leq k\leq n_1$. Thus $u_1\in V_{\sigma_1}(n_1,n_2)$ and 
$(v_1,v_2)\in V_{(\sigma_1, \sigma_2)}(n_1,n_2,H)$.
Assume that the property is true for $q-1$, $q\geq 3$ and let 
$$(v_1,v_2,\dots,v_q) \in  V(n_1,n_2,\dots,n_q,H)\cap  \prod^{q}_{k=1}V_{\sigma_q\dots\sigma_{k+1}\sigma_k}(n_k,H)$$ 
for some $\sigma=(\sigma_1,\sigma_2,\dots,\sigma_q)\in S(n_1,n_2,\dots,n_q,n)$. Let $$(u_1,u_2,\dots,u_q) \in V(n_1,n_2)\times V(n_2,n_3)\times \dots \times V(n_q,H)$$ be such that $v_k=u_q\dots u_{k+1}u_k$, $1\leq k\leq q.$
$$(v_2,v_3,\dots,v_q) \in  V(n_2,n_3,\dots,n_q,H)\cap  \prod^{q}_{k=2}V_{\sigma_q\dots\sigma_{k+1}\sigma_k}(n_k,H).$$
By the induction assumption 
$$(v_2,v_3,\dots,v_q) \in  V_{(\sigma_2,\sigma_3,\dots\sigma_q)}(n_2,n_3,\dots,n_q,H)$$  i.e.

$$(u_2,u_3,\dots,u_q)\in V_{\sigma_2}(n_2,n_3)\times V_{\sigma_3}(n_{3},n_{4})\times \dots\times{V_{\sigma_q}(n_q,H)}$$

 Furthermore,
$(v_1,v_2) \in V(n_1,n_2,H)\cap V_{\sigma_q\dots\sigma_{2}\sigma_1}(n_1,H) \times V_{\sigma_q\dots\sigma_{3}\sigma_2}(n_2,H)$. By the induction assumption
 $$(v_1,v_2)\in V_{(\sigma_1,\sigma_q\dots\sigma_{3}\sigma_2)}(n_1,n_2,H)$$
i.e.
$$\left\{
\begin{array}{l}
u_1\in V_{\sigma_1}(n_1,n_2) \\
 u_q\dots u_3 u_2 \in V_{ \sigma_q\dots\sigma_{3}\sigma_2}(n_2,H)
\end{array}
\right.
$$
 It follows that
$$(u_1,u_2,\dots,u_q)\in \cup_{k=1}^q {V_{\sigma_1}(n_1,n_2)}\times V_{\sigma_2}(n_{2},n_{3})\times
\dots\times{V_{\sigma_q}(n_q,H)}$$
and 
$$(v_1,v_2,\dots,v_q) \in  V_{(\sigma_1,\sigma_2,\dots\sigma_q)}(n_1,n_2,\dots,n_q,H)$$
 as desired.
\item 
Let $(v_1,v_2,\dots,v_q)\in \partial V_{\sigma}(n_1,n_2,\dots,n_q,H)$. By Lemma \ref{l7}, there exists $k\in \{1,2,...,q\}$ 
and $$(u_1,u_2,\dots,u_q)\in \cup_{k=1}^q\overline{V_{\sigma_1}(n_1,n_2)}\times\dots\times\partial V_{\sigma_k}(n_{k},n_{k+1})\times
\dots\times\overline{V_{\sigma_q}(n_q,H)}$$
such that
$$\psi(n_1,n_2,\dots,n_q,H)(u_1,u_2,\dots,u_q)=(v_1,v_2,\dots,v_q)$$
Define
$$l=\left\{\begin{array}{ll}
       q&\mbox{if } u_q\in \partial V_{\sigma_q}(n_q,H)\\
			\mbox{max}\{i\in\{1,\dots,q-1\}|u_i\in\partial V_{\sigma_i}(n_i,n_{i+1})\}
			&\mbox{if } u_q\notin\partial V_{\sigma_q}(n_q,H).
			\end{array}
			\right.
			$$
Then $(u_l,u_{l+1},\dots,u_q)\in\partial V_{\sigma_l}(n_l,n_{l+1})\times V_{\sigma_{l+1}(n_{l+1},n_{l+2}})\times\dots
\times V_{\sigma_q}(n_q,H)$.
By Lemma \ref{l8}.  $v_l=u_q\dots u_{l+1}u_l\in\partial V_{\sigma_q\dots\sigma_{l+1}\sigma_l}(n_l,H)$, therefore \\ 
$(v_1,\dots,v_l,\dots,v_q)$ is in $$V(n_1,H)\times\dots\times V(n_{l-1},H)\times\partial V_{\sigma_q\dots\sigma_{l+1}\sigma_l}(n_l,H)\times V(n_{l+1},H)\times\dots\times V(n_q,H) $$ as desired. 
\end{enumerate}
\end{proof}
\section{The cell structure for flag manifolds}\label{s6}	
Let $H$ be a finite dimensional inner product vector space of dimension $n$. $(n_k)_{1\leq k\leq q}$ a sequence of integers such that $1 \leq n_1 \prec n_2 \dots \prec n_{q} \prec n$.
Define the flag manifold $G(n_1,\dots,n_q,H)$ to be the subspace of $\prod^{q}_{k=1}G(n_k,H)$ 
given by
\begin{equation}\label{eq6} 
G(n_1,\dots,n_q,H)=\{{(X^k)_{1\leq k\leq q}}\in \prod^{q}_{k=1}G(n_k,H) \suchthat X^1\subset X^2\subset\dots\subset X^{q}\}
\end{equation}
$G(n_1,\dots,n_q,H)$ is a subspace of a Haudorff space, it is then Hausdorff. 
Let 
\begin{equation}\label{eq17} 
\prod^{q}_{k=1}p(n_k,H):\prod^{q}_{k=1}V(n_k,H)\longrightarrow \prod^{q}_{k=1}G(n_k,H)
\end{equation}
be the product of the quotient maps defined by (\ref{eq1}). Then
\begin{equation}\label{eq18}  
G(n_1,\dots,n_q,H)=\prod^{q}_{k=1}p(n_k,H)(V(n_1,\dots,n_q,H))
\end{equation}
$V(n_1,\dots,n_q,H)$ is compact, then so is $G(n_1,\dots,n_q,H)$. Therefore the map $\prod^{q}_{k=1}p(n_k,H)$ induces a quotient map
\begin{equation}\label{eq7} 
 p(n_1,\dots,n_q,H):V(n_1,\dots,n_q,H)\longrightarrow G(n_1,\dots,n_q,H) 
\end{equation}
\\

%%%%%%%%%%%%%%%%%%%%%%%%%%%%%%%%%%%%%%%%%%%%%%%%%%%%%%%%%%%%%%%%%%%%%%%%%%%%%%%%%%%%%%%%%%%%%%%%%%%%%%%%%%%%%%%%%%%%%%%%%%%%%%%%%%%%%%%%%%%%%%%%%%%%%%%%%%%%%%%%%%%%%%%%%%%%%%%%%%%%%%%%%%%%%%%%%%%%%%%%%%%%%%%%%%%%%%%%%%%%%%%%%%%%%%%%%%%%%%%%%%%%%%%%%%%%%%%%%%%%%%%%%%%%%%%%%%%%%%%%%%%%%%%%%%%%%%%%%%%%%%%%%%%%
Let $(h_i)_{1\leq i\leq n}$ be an orthonormal basis of $H$, $H_0$ the zero subspace of $H$ and $H_{k}$ the subspace spanned by $(h_i)_{1\leq i\leq k}$, $1\leq k\leq n$. $(H_k)_{0\leq k\leq n}$ is  a complete flag in $H$.
Let $(X^1,X^2,\dots,X^{q}) \in G(n_1,\dots,n_q,H)$, as explained in (\ref{eq11}), the complete flag $(H_k)_{0\leq k\leq n}$ in $H$ induces complete flags
$$X^i_{0}\subset X^i_{1}\subset \dots \subset X^i_{n_{i}} = X^i, \quad 1\leq i\leq q.$$
$X^q \in G(n_q,H)$, let  $\sigma_{X^q}$ be its Schubert symbol as defined in (\ref{eq9}). Similarly for $1\leq i\prec q $, $X^i\in G(n_i,X^{i+1})$, let  $\sigma_{X^i}$ be its Schubert symbol and define the  Schubert symbol $\sigma_{(X^1,X^2,\dots,X^{q})}$ of the flag $(X^1,X^2,\dots,X^{q}) \in G(n_1,\dots,n_q,H)$ to be the general Schubert symbol
\begin{equation}\label{eq15}
\sigma_{(X^1,X^2,\dots,X^{q})}=(\sigma_{X^1},\sigma_{X^2},\dots,\sigma_{X^q}) \in S(n_1,n_2,\dots,n_q,n).
\end{equation}  
 For $\sigma=(\sigma_1,\sigma_2,\dots,\sigma_{q}) \in S(n_1,n_2,\dots,n_q,n)$, define 
$$G_{\sigma}(n_1,\dots,n_q,H)=\{(X^1,X^2,\dots,X^{q})\in G(n_1,\dots,n_q,H)\suchthat \sigma_{(X^1,X^2,\dots,X^{q})}=\sigma \}$$
The following lemma is an immediate consequence of the previous definition. 
\begin{lem}\label{l6} The subsets $G_{\sigma}(n_1,\dots,n_q,H)$, $\sigma \in S(n_1,n_2,\dots,n_q,n)$ form a partition of $G(n_1,\dots,n_q,H)$.
 \end{lem} 
%%%%%%%%%%%%%%%%%%%%%%%%%%%%%%%%%%%%%%%%%%%%%%%%%%%%%%%%%%%%%%%%%%%%%%%%%%%%%%%%%%%%%%%%%%%%%%%%%%%%%%%%%%%%%%%%%%%%%%%%%%%%%%%%%%%%%%%%%%%%%%%%%%%%%%%%%%%%%%%%%%%%%%%%%%%%%%%%%%%%%%%%%%%%%%%%%%%%%%%%%%%%%%%%%%%%%%%%%%%%%%%%%%%%%%%%%%%%%%%%%%%%%%%%%%%%%%%%%%%%%%%%%%%%%%%%%%%%%%%%%%%%%%%%%%%%%%%%%%%%%%%%%%%%
%
%
%
%

\begin{lem}\label{l11} Let $\sigma =(\sigma_1,\sigma_2,...,\sigma_q) \in S(n_1,n_2,...,n_{q},n)$, then
\begin{enumerate}
	\item $p(n_{1}, n_{2},...,n_{q},H)_{/}:V_{\sigma}(n_{1}, n_{2},...,n_{q},H)\to G(n_{1}, n_{2},...,n_{q},H)$ is injective.
	\item $p(n_{1}, n_{2},...,n_{q},H)(V_{\sigma}(n_{1}, n_{2},...,n_{q},H)) \\ \cap p(n_{1}, n_{2},...,n_{q},H)(\partial V_{\sigma}(n_{1}, n_{2},...,n_{q},H))=\emptyset.$
\end{enumerate}
\end{lem}
\begin{proof}
The first claim is a consequence of the first assertion of lemma \ref{l4} and the first assertion of Lemma \ref{l9} while the second claim is a consequence of the second assertions of the same Lemmas.
	\end{proof}
%%%%%%%%%%%%%%%%%%%%%%%%%%%%%%%%%%%%%%%%%%%%%%%%%%%%%%%%%%%%%%%%%%%%%%%%%%%%%%%%%%%%%%%%%%%%%%%%%%%%%%%%%%%%%%%%%%%%%%%%%%%%%%%%%%%%%%%%%%
%%%%%%%%%%%%%%%%%%%%%%%%%%%%%%%%%%%%%%%%%%%%%%%%%%%%%%%%%%%%%%%%%%%%%%%%%%%%%%%%%%%%%%%%%%%%%%%%%%%%%%%%%%%%%%%%%%%%%%%%%%%%%%%%%%%%%%%%%%
%%%%%%%%%%%%%%%%%%%%%%%%%%%%%%%%%%%%%%%%%%%%%%%%%%%%%%%%%%%%%%%%%%%%%%%%%%%%%%%%%%%%%%%%%%%%%%%%%%%%%%%%%%%%%%%%%%%%%%%%%%%%%%%%%%%%%%%%%%%%%%%%%%%%%%%%%%%%%%%%%%%%%%%%%%%%%%%%%%%%%%%%%%%%%%%%%%%%%%%%%%%%%%%
\begin{lem}\label{l12}
Let $\sigma=(\sigma_1,\sigma_2,\dots,\sigma_q)\in S(n_1,n_2,\dots,n_q,n)$ then
  $$G_{\sigma}(n_1,n_2,\dots,n_q,H)=G(n_1,n_2,\dots,n_q,H)\cap  \prod^{q}_{k=1}
	G_{\sigma_q\dots\sigma_{k+1}\sigma_k}(n_k,H).$$
\end{lem}
\begin{proof}
By Lemma \ref{l3} 
$$G_{\sigma}(n_1,n_2,\dots,n_q,H)\subset G(n_1,n_2,\dots,n_q,H)\cap  \prod^{q}_{k=1}
	G_{\sigma_q\dots\sigma_{k+1}\sigma_k}(n_k,H)$$
	therefore by Lemma \ref{l6}, the subsets $$G(n_1,n_2,\dots,n_q,H)\cap  \prod^{q}_{k=1}
	G_{\sigma_q\dots\sigma_{k+1}\sigma_k}(n_k,H)$$  
	cover $G(n_1,n_2,\dots,n_q,H)$. By lemma \ref{l2}
	the subsets $$G(n_1,n_2,\dots,n_q,H)\cap  \prod^{q}_{k=1}
	G_{\sigma_q\dots\sigma_{k+1}\sigma_k}(n_k,H)$$
	are disjoint, therefore $$G_{\sigma}(n_1,n_2,\dots,n_q,H)=G(n_1,n_2,\dots,n_q,H)\cap  \prod^{q}_{k=1}
	G_{\sigma_q\dots\sigma_{k+1}\sigma_k}(n_k,H)$$
\end{proof}
%%%%%%%%%%%%%%%%%%%%%%%%%%%%%%%%%%%%%%%%%%%%%%%%%%%%%%%%%%%%%%%%%%%%%%%%%%%%%%%%%%%%%%%%%%%%%%%%%%%%%%%%%%%%%%%%%%%%%%%%%%%%%%%%%%%%%%%%%%
%%%%%%%%%%%%%%%%%%%%%%%%%%%%%%%%%%%%%%%%%%%%%%%%%%%%%%%%%%%%%%%%%%%%%%%%%%%%%%%%%%%%%%%%%%%%%%%%%%%%%%%%%%%%%%%%%%%%%%%%%%%%%%%%%%%%%%%%%%
%%%%%%%%%%%%%%%%%%%%%%%%%%%%%%%%%%%%%%%%%%%%%%%%%%%%%%%%%%%%%%%%%%%%%%%%%%%%%%%%%%%%%%%%%%%%%%%%%%%%%%%%%%%%%%%%%%%%%%%%%%%%%%%%%%%%%%%%%%%%%%%%%%%%%%%%%%%%%%%%%%%%%%%%%%%%%%%%%%%%%%%%%%%%%%%%%%%%%%%%%%%%%%%
\begin{lem}\label{l13} Let $\sigma =(\sigma_1,\sigma_2,...,\sigma_q) \in S(n_1,n_2,...,n_{q},n)$, then
$$p(n_{1}, n_{2},...,n_{q},H)(V_{\sigma}(n_{1}, n_{2},...,n_{q},H))=G_{\sigma}(n_{1}, n_{2},...,n_{q},H).$$
	\end{lem}
\begin{proof} By Lemma \ref{l9}
 $$V_{\sigma}(n_1,n_2,\dots,n_q,H)=V(n_1,n_2,\dots,n_q,H)\cap  \prod^{q}_{k=1}
	V_{\sigma_q\dots\sigma_{k+1}\sigma_k}(n_k,H)$$
	thus
	$$\begin{array}{rcll}
  && p(n_{1}, n_{2},...,n_{q},H)(V_{\sigma}(n_{1}, n_{2},...,n_{q},H))\\
  &\subset&G(n_1,n_2,\dots,n_q,H)\cap  \prod^{q}_{k=1} G_{\sigma_q\dots\sigma_{k+1}\sigma_k}(n_k,H) & (\mbox{by Lemma }  \ref{l4})\\
	&\subset&G_{\sigma}(n_{1}, n_{2},...,n_{q},H) & (\mbox{by Lemma } \ref{l12}) 
\end{array}$$
Let $(X^1,X^2,\dots,X^{q}) \in G_{\sigma}(n_1,\dots,n_q,H)$. By Lemma \ref{l12}, $(X^1,X^2,\dots,X^{q})$ is in 	$G(n_1,n_2,\dots,n_q,H)\cap  \prod^{q}_{k=1} G_{\sigma_q\dots\sigma_{k+1}\sigma_k}(n_k,H)$ and by Lemma \ref{l4}, there exists $$(v_1,v_2,\dots,v_{q})\in \prod^{q}_{k=1} V_{\sigma_q\dots\sigma_{k+1}\sigma_k}(n_k,H)$$ such that $$\prod^{q}_{k=1}p(n_k,H)(v_1,v_2,\dots,v_{q})=(X^1,X^2,\dots,X^{q}).$$
$(X^1,X^2,\dots,X^{q})\in G(n_1,\dots,n_q,H)$, thus $X^1\subset X^2 \subset \dots \subset X^{q}$. By (\ref{eq4}) $$(v_1,v_2,\dots,v_{q})\in V(n_{1}, n_{2},...,n_{q},H)$$ thus
$$(v_1,v_2,\dots,v_{q})\in V(n_1,n_2,\dots,n_q,H)\cap  \prod^{q}_{k=1}
	V_{\sigma_q\dots\sigma_{k+1}\sigma_k}(n_k,H)$$
By Lemma \ref{l9}, $(v_1,v_2,\dots,v_{q})\in V_{\sigma}(n_{1}, n_{2},...,n_{q},H)$ and $$p(n_{1}, n_{2},...,n_{q},H)(V_{\sigma}(n_{1}, n_{2},...,n_{q},H))=G_{\sigma}(n_{1}, n_{2},...,n_{q},H).$$ 
\end{proof}
%%%%%%%%%%%%%%%%%%%%%%%%%%%%%%%%%%%%%%%%%%%%%%%%%%%%%%%%%%%%%%%%%%%%%%%%%%%%%%%%%%%%%%%%%%%%%%%%%%%%%%%%%%%%%%%%%%%%%%%%%%%%%%%%%%%%%%%%%%
%%%%%%%%%%%%%%%%%%%%%%%%%%%%%%%%%%%%%%%%%%%%%%%%%%%%%%%%%%%%%%%%%%%%%%%%%%%%%%%%%%%%%%%%%%%%%%%%%%%%%%%%%%%%%%%%%%%%%%%%%%%%%%%%%%%%%%%%%%
%%%%%%%%%%%%%%%%%%%%%%%%%%%%%%%%%%%%%%%%%%%%%%%%%%%%%%%%%%%%%%%%%%%%%%%%%%%%%%%%%%%%%%%%%%%%%%%%%%%%%%%%%%%%%%%%%%%%%%%%%%%%%%%%%%%%%%%%%%%%%%%%%%%%%%%%%%%%%%%%%%%%%%%%%%%%%%%%%%%%%%%%%%%%%%%%%%%%%%%%%%%%%%%
\begin{lem}\label{l14}
Let $\sigma\in S(n_1,n_2,\dots,n_q,H)$, then
 $G_{\sigma}(n_1,\dots,n_q,H)$ is an open cell in $G(n_1,\dots,n_q,H)$ of dimension $\dim_{\R} K.d(\sigma)$. Furthermore
	$$\partial G_{\sigma}(n_1,\dots,n_q,H)=p(n_{1}, n_{2},...,n_{q},H)(\partial V_{\sigma}(n_1,\dots,n_q,H)).$$
\end{lem}
\begin{proof}
 Immediate consequence of Lemma \ref{l11}, Lemma \ref{l13} and lemma \ref{l19}. 
\end{proof}
\begin{lem}\label{l24}
Let $\sigma=(\sigma_1,\sigma_2,\dots,\sigma_q)\in S(n_1,n_2,\dots,n_q,n)$, $Y=(Y_1,Y_2,\dots,Y_q)\in\partial G_{\sigma}(n_1,n_2,\dots,n_q,H)$ and $\tau=(\tau_1,\tau_2,\dots,\tau_q)$ the Schubert symbol of $(Y_1,Y_2,\dots,Y_q)$ then
\begin{enumerate}
	\item $\tau_q\dots\tau_{k+1}\tau_k\leq\sigma_q\dots\sigma_{k+1}\sigma_k, 1\leq k\leq q$.
	\item $\tau_q\dots\tau_{k_0+1}\tau_{k_0}<\sigma_q\dots\sigma_{k_0+1}\sigma_{k_0}$ for some $k_0 \in \{1,2,\dots ,q \}$.
	\item $d(\tau)<d(\sigma)$.
\end{enumerate}
\end{lem}
\begin{proof}
\begin{enumerate}
	\item \begin{eqnarray*}
	      Y=(Y_1,Y_2,\dots,Y_q)&\subset& \overline{G_{\sigma}(n_1,n_2,\dots,n_q,H)}\\
	                   &\subset&\overline{\prod_{k=1}^q G_{\sigma_q\dots\sigma_{k+1}\sigma_k}(n_k,H)} \qquad (\mbox{by \ref{l12}})\\
						      					&=&\prod_{k=1}^q\overline{G_{\sigma_q\dots\sigma_{k+1}\sigma_k}(n_k,H)}\\
			\end{eqnarray*}
					thus $Y_k\in \overline{G_{\sigma_q\dots\sigma_{k+1}\sigma_k}(n_k,H)}, 1\leq k\leq q.$
					\begin{eqnarray*}
					Y=(Y_1,Y_2,\dots,Y_q)&\in& G_{\tau}(n_1,n_2,\dots,n_q,H)\\
					                       &\subset&\prod_{k=1}^{q}G_{\tau_q\dots\tau_{k+1}\tau_k}(n_k,H) \qquad (\mbox{by \ref{l12}})
	\end{eqnarray*}		
It follows that 
$$\left\{
\begin{array}{l}
Y_k\in \overline{G_{\sigma_q\dots\sigma_{k+1}\sigma_k}(n_k,H)}\\
Y_k\in G_{\tau_q\dots\tau_{k+1}\tau_k}(n_k,H)
\end{array}
\right.
$$
By Lemma \ref{l5},
			$\tau_q\dots\tau_{k+1}\tau_k\leq \sigma_q\dots\sigma_{k+1}\sigma_k$, $1\leq k\leq q$.
			\item $Y\in\partial G_{\sigma}(n_1,n_2,\dots,n_q,H)$, thus
			$Y\notin G_{\sigma}(n_1,n_2,\dots,n_q,H)$. There exists $k_0\in\{1,\dots,q\}$ such that
			$Y_{k_0}\notin G_{\sigma_q\dots\sigma_{k_0+1}\sigma_{k_0}}(n_{k_0},H),$ thus
			$Y_{k_0}\in\partial G_{\sigma_q\dots\sigma_{k_0+1}\sigma_{k_0}}(n_{k_0},H)$. By Lemma \ref{l4}
			$\tau_q\dots\tau_{k_0+1}\tau_{k_0}<\sigma_q\dots\sigma_{k_0+1}\sigma_{k_0}$.
\item The proof of the fact that $d(\tau)<d(\sigma)$ is by induction on $q$ and uses the previous two properties.		
\end{enumerate}
\end{proof}		
%%%%%%%%%%%%%%%%%%%%%%%%%%%%%%%%%%%%%%%%%%%%%%%%%%%%%%%%%%%%%%%%%%%%%%%%%%%%%%%%%%%%%%%%%%%%%%%%%%%%%%%%%%%%%%%%%%%%%%%%%%%%%%%%%%%%%%%%%%
%%%%%%%%%%%%%%%%%%%%%%%%%%%%%%%%%%%%%%%%%%%%%%%%%%%%%%%%%%%%%%%%%%%%%
%%%%%%%%%%%%%%%%%%%%%%%%%%%%%%%%%%%%%%%%%%%%%%%%%%%%%%%%%%%%%%%%%%%%%
%%%%%%%%%%%%%%%%%%%%%%%%%%%%%%%%%%%%%%%%%%%%%%%%%%%%%%%%%%%%%%%%%%%%%%%%%%%%%%%%%%%%%%%%%%%%%%%%%%%%%%%%%%%%%%%%%%%%%%%%%%%%%%%%%%%%%%%%%%%%%%%%%%%%%%%%%%%%%%%%%%%%%%%%%%%%%%%%%%%%%%%%%%%%%%%%%%%%%%%%%%%%%%%
\begin{thm}\label{t2}
There is a CW-complex structure on $G
(n_1,n_2,\dots,n_q,H)$ with open cells
$G_{\sigma}(n_1,n_2,\dots,n_q,H),  \sigma\in S(n_1,n_2,\dots,n_q,H)$.
\end{thm}
\begin{proof}
immediate consequence of lemmas \ref{l22}, \ref{l6}, \ref{l14} and \ref{l24}.
\end{proof}
\begin{cor}\label{c2}
$G(n_1,n_2,\dots,n_q,H)$ is a topological manifold of dimension
 $$((n_1n_2+n_2n_3+\dots+n_{q-1}n_q+n_qn)-({n_1}^{2}+{n_2}^{2}+\dots+{n_q}^{2})).\dim_{\R}K.$$

\end{cor}
\begin{proof}
$G(n_1,n_2,\dots,n_q,H)$ is a Hausdorff. Let $X_0=(X_0^1,X_0^2,\dots,X_0^q)\in G(n_1,n_2,\dots,n_q,H)$ and construct
by induction an orthonormal basis for $H$ such that
$(h_{n-n_k+1},\dots,h_{n-1},h_n)$ is a basis for $X_0^k$, $1\leq k\leq q$. Let 
$\sigma=(\sigma_1,\sigma_2,\dots,\sigma_q)\in S(n_1,n_2,\dots,n_q,n)$ be the general Schubert symbol given by
$$\sigma_k(i)=\left\{\begin{array}{ll}
n-n_q+i& \mbox{ for } k= q \mbox{ and } 1\leq i\leq n_q\\
n_{k+1}-n_k+i& \mbox{ for } 1\leq k\prec q \mbox{ and } 1\leq i\leq n_k
\end{array}
\right.
$$
It is easily checked  that 
$X_0\in G_{\sigma}(n_1,n_2,\dots,n_q,H).$\\
$G_{\sigma}(n_1,n_2,\dots,n_q,H)$ is an open cell in $G(n_1,n_2,\dots,n_q,H)$ of top dimension. It is therefore an open neighborhood of $X_0$ which is homeomorphic to an Euclidean space. $G(n_1,n_2,\dots,n_q,H)$ is a topological manifold of dimension
\begin{eqnarray*}
\dim _{\R}K.d(\sigma)&=&\dim _{\R}K.\sum_{k=1}^qd(\sigma_k)\\
                           &=&\dim _{\R}K.(n_1(n_2-n_1)+n_2(n_3-n_2)+\\
													&&\dots+n_{q-1}(n_q-n_{q-1})+n_q(n-n_q))\\
													&=&\dim_{\R}K.((n_1n_2+n_2n_3+\dots+n_{q-1}n_q+n_qn)
\\
													&&-({n_1}^{2}+{n_2}^{2}+\dots+{n_q}^{2})).
\end{eqnarray*}
\end{proof}

%%%%%%%%%%%%%%%%%%%%%%%%%%%%%%%%%%%%%%%%%%%%%%%%%%%%%%%%%%%%%%%%%%%%%%%%%%%%%%%%%%%%%%%%%%%%%%%%%%%%%%%%%%%%%%%%%%%%%%%%%%%%%%%%%%%
%%%%%%%%%%%%%%%%%%%%%%%%%%%%%%%%%%%%%%%%%%%%%%%%%%%%%%%%%%%%%%%%%%%%%%%%%%%%%%%%%%%%%%%%%%%%%%%%%%%%%%%%%%%%%%%%%%%%%%%%%%%%%%%%%%%
%%%%%%%%%%%%%%%%%%%%%%%%%%%%%%%%%%%%%%%%%%%%%%%%%%%%%%%%%%%%%%%%%%%%%%%%%%%%%%%%%%%%%%%%%%%%%%%%%%%%%%%%%%%%%%%%%%%%%%%%%%%%%%%%%%%%%%%%%%%%%%%%%%%%%%%%%%%%%%%%%%%%%%%%%%%%%%%%%%%%%%%%%%%%%%%%%%%%%%%%%%%%%%%
\begin{rem}
\begin{enumerate}
\item For $K=\C$ or $K=\H$, all cells in $G(n_1,n_2,\dots,n_q,H)$ are even dimensional. Consequently, they are all cycles and they form a basis for the integral homology groups of $G(n_1,n_2,\dots,n_q,H)$. 
	\item Let $H_{\C}$ (resp. $H_{\H}$) be an $n$-dimensional $\C$ vector space (resp. $H_{\H}$ vector space). There is a one-to-one correspondence between cells of $G(n_1,n_2,\dots,n_q,H_{\C})$ of dimension $2i$ and cells of $G(n_1,n_2,\dots,n_q,H_{\H})$ of dimension $4i$. In particular  $P_{\H}(t)= P_{\C}(t^{2})$ where $P_{\C}$ (resp. $P_{\H}$) is the Poincar\'e polynomial of $G(n_1,n_2,\dots,n_q,H_{\C})$ (resp. $G(n_1,n_2,\dots,n_q,H_{\H})$. It follows that $P_{\H}(1)= P_{\C}(1)$ and the two manifolds
	$G(n_1,n_2,\dots,n_q,H_{\C})$ and $G(n_1,n_2,\dots,n_q,H_{\H})$
have the same Euler characteristic. 
\end{enumerate} \color[rgb]{0,0,0}
\end{rem}
%%%%%%%%%%%%%%%%%%%%%%%%%%%%%%%%%%%%%%%%%%%%%%%%%%%%%%%%%%%%%%%%%%%%%%%%%%%%%%%%%%%%%%%%%%%%%%%%%%%%%%%%%%%%%%%%%%%%%%%%%%%%%%%%%%%%%%%%%%%%%%%%%%%%%%%%%%%%%%%%%%%%%%%%%%%%%%%%%%%%%%%%%%%%%%%%%%%%%%%%%%%%%%%%%%%%%%%%%%%%%%%%%%%%%%%%%%%%%%%%%%%%%%%%%%%%%%%%%%%%%%%%%%%%%%%%%%%%%%%%%%%%%%%%%%%%%%%%%%%%%%%%%%%%%%%%%%%%%%%%%%%%%%%%%%%%%%%%%%%%%%%%%%%%%%%%%%%%%%%%%%%%%%%%%%%%%%%%%%%%%%%%%%%%%%%%%%%%%%%%%%%%%%%%%%%%%%%%%%%%%%%%%%%%%%%%%%%%%%%%%%%%%%%%%%%%%%%%%%%%%%%%%%%%%%%%%%%%%%%%%%%

\end{document}